\documentclass[11pt,a4paper]{article}
\usepackage[T1]{fontenc}
\usepackage{lmodern}
\usepackage[margin=25mm]{geometry}
\usepackage{amsmath,amssymb,amsthm,mathtools}
\usepackage{microtype,xcolor}
\usepackage[colorlinks=true,allcolors=blue!45!black]{hyperref}
\newtheorem{theorem}{Theorem}
\newtheorem{proposition}{Proposition}
\newtheorem{lemma}{Lemma}
\newcommand{\R}{\mathbb R}
\newcommand{\ind}{\mathbf 1}
\newcommand{\intr}{\operatorname{int}}
\newcommand{\conv}{\operatorname{conv}}
\newcommand{\Prob}{\mathcal P}
\newcommand{\If}{I^*}
\hypersetup{
  pdftitle={Fractional illumination and the optimal exponential rate in Hadwiger's covering conjecture},
  pdfauthor={Yegor Gorodzha}
}

\title{Fractional illumination and the optimal exponential rate in Hadwiger's covering conjecture\thanks{The main argument was found in conversation with OpenAI's GPT-6 Pro. See Appendix~\ref{app:ai} for an account of the discovery and AI use.}
\thanks{The author is grateful to Gil Kur and Steven Hoehner for their help in presenting this work.}}
\author{Yegor Gorodzha\thanks{ETH Zurich, Switzerland. The author was supported during a scientific internship at the Institute of Science and Technology Austria (ISTA) by the European Research Council (ERC), grant INF\_2 (grant agreement no.~101161364).}}
\date{}

\makeatletter
\renewcommand{\@maketitle}{%
  \newpage\null
  \vskip .5em
  \begin{center}%
    \let\footnote\thanks
    {\LARGE\@title\par}%
    \vskip .75em
    {\large\lineskip .5em
      \begin{tabular}[t]{c}\@author\end{tabular}\par}%
  \end{center}%
  \par\vskip .25em}
\makeatother

\begin{document}
\maketitle
\begin{abstract}
We show that the fractional illumination number introduced by
Nasz\'odi~\cite{Naszodi} is at most $2^d$ for every convex body in $\R^d$,
with equality exactly for parallelotopes. We also prove that
every such body can be covered by at most
$2^d(d\log d+d\log\log d+O(d))$ smaller positive homothetic copies, establishing
the optimal exponential rate in Hadwiger's covering conjecture. The proofs
use a covering measure obtained by minimizing an overlap energy and a
greedy covering argument on a finite net.
\end{abstract}

\section{Introduction}

\emph{Hadwiger's covering conjecture}, posed in 1957~\cite{Hadwiger},
asks whether every convex body $K\subset\R^d$ can be covered by $2^d$
smaller positive homothets, that is, sets $z+\lambda K$ with $0<\lambda<1$.
Levi~\cite{Levi} had settled the planar case in 1955. Gohberg and
Markus~\cite{GohbergMarkus} independently posed the homothetic formulation.
Write $H(K)$ for the least number of such homothetic copies covering $K$ and
$H(d)=\sup_K H(K)$, where the supremum is over convex bodies in $\R^d$.

Rogers' covering theorem~\cite{Rogers} and the Rogers--Shephard
inequality~\cite{RogersShephard} give the classical bound
$H(d)=O(4^d\sqrt d\log d)$, see Rogers and Zong~\cite{RogersZong}.
Huang, Slomka, Tkocz, and Vritsiou~\cite[Theorem~1.1]{HSTV} improved this to
$4^d e^{-c\sqrt d}$, and Campos, van Hintum, Morris, and
Tiba~\cite{CHMT} obtained $4^d e^{-cd/\log d}$.
Their estimate in terms of the isotropic constant~\cite{CHMT},
together with the solution of the slicing problem by Klartag and
Lehec~\cite[Theorem~1.2]{KL}, yields $H(d)\leq(4-\varepsilon)^d$.
These bounds hold for all sufficiently large $d$, with absolute
constants $c,\varepsilon>0$.

In the equivalent illumination formulation, a direction $u\in S^{d-1}$
\emph{illuminates} $x\in\partial K$ if $x+tu\in\intr K$ for some $t>0$.
The least number $I(K)$ of directions illuminating every boundary point
equals $H(K)$, by Boltyanski~\cite{Boltyanski}.
Nasz\'odi~\cite[Definition~1]{Naszodi} introduced the \emph{fractional
illumination number} $\If(K)$ in 2009: it is the infimum of $\sum_jw_j$
over finite families $u_1,\dots,u_m\in S^{d-1}$ and $w_1,\dots,w_m\geq0$ satisfying
\[
 \sum_{j:\,u_j\text{ illuminates }x}w_j\geq1
 \qquad(x\in\partial K).
\]
Nasz\'odi proved $\If(K)\leq\binom{2d}{d}$ in general and $\If(K)\leq2^d$ for
centrally symmetric bodies~\cite[Theorem~2 and Corollary~3]{Naszodi}.
Artstein-Avidan and Slomka~\cite{ArtsteinSlomka} characterized
equality in the symmetric case. Our first theorem below proves Nasz\'odi's
conjecture~\cite[Conjecture~7]{Naszodi} for all convex bodies. The second result gives the optimal exponential rate for $H(d)$.

\begin{theorem}\label{thm:fractional}
Every convex body $K\subset\R^d$ satisfies $\If(K)\leq2^d$.
Equality holds if and only if $K$ is a parallelotope, that is, an invertible
affine image of $[-1,1]^d$.
\end{theorem}

\begin{theorem}\label{thm:hadwiger}
As $d\to\infty$, the covering number $H(d)$ satisfies
\[
 H(d)\leq2^d\bigl(d\log d+d\log\log d+O(d)\bigr),
\]
with an absolute implicit constant. In particular,
$H(d)=2^{d+O(\log d)}$.
\end{theorem}

\section{A covering measure from overlap energy}

A \emph{convex body} is a compact convex subset of $\R^d$ with nonempty
interior. For compact sets $A,T\subset\R^d$, let $N(A,T)$ be the least
number of translates of $T$ covering $A$, with $N(A,T)=\infty$ if no
finite cover exists. Thus $H(K)=\min_{0<\lambda<1}N(K,\lambda K)$.
Both $H(K)$ and $\If(K)$ are invariant under invertible affine maps:
apply the map to each homothet, or its linear part to each illumination
direction and normalize to unit length. Write $|A|$ for the Lebesgue volume
of a measurable set $A\subset\R^d$,
$\ind_A$ for its indicator, $\#A$ for the cardinality of a finite set,
and $B_2^d$ for the closed Euclidean unit ball.
For sets $A,T$, write $A+T=\{a+t:a\in A,t\in T\}$ and $A-T=A+(-T)$.

For a finite nonnegative Borel measure $\nu$, set
\[
 (\nu*\ind_T)(x)=\int\ind_T(x-y)\,d\nu(y).
\]
Following Artstein-Avidan and Slomka~\cite{ArtsteinSlomka},
we say that $\nu$ \emph{covers $A$ by translates of $T$} if this quantity
is at least one for every $x\in A$. In particular,
$\sum_jw_j\delta_{z_j}$ represents a finite family of weighted translates $z_j+T$. Here
$\delta_z$ denotes unit mass at $z$ and coverage is always understood pointwise.

For compact $T$ of positive volume, its \emph{covariogram}
(Matheron~\cite[Section~4.3]{Matheron}) is
\[
 g_T(z)=|T\cap(T+z)|
 =\int_{\R^d}\ind_T(x)\ind_T(x-z)\,dx.
\]
It is even and continuous, since
\[
 |g_T(z+h)-g_T(z)|
 \leq\|\ind_T(\,\cdot-h)-\ind_T\|_{L^1}\longrightarrow0
 \qquad(h\to0).
\]
For nonempty compact $A$, let $\Prob(A)$ denote the Borel probability
measures supported on $A$, and define
\begin{equation}\label{eq:energy}
 e(A,T)=\min_{\mu\in\Prob(A)}
       \iint g_T(a-b)\,d\mu(a)\,d\mu(b).
\end{equation}
The minimum exists by weak-$*$ compactness of $\Prob(A)$ and continuity of the
kernel on $A\times A$. This is the usual minimum-energy problem for the
kernel $g_T(a-b)$ in potential theory~\cite[Section~2]{Fuglede}.

\begin{proposition}\label{prop:cover}
For nonempty compact $A\subset\R^d$ and compact $T\subset\R^d$ with
$|T|>0$, there is an absolutely continuous measure $\nu$ covering $A$ by
translates of $T$, supported on $A-T$, with
\begin{equation}\label{eq:cover}
 \nu(\R^d)=\frac{|T|}{e(A,T)}\leq\frac{|A+T|}{|T|}.
\end{equation}
\end{proposition}

\begin{proof}
Let $\mu$ attain the minimum in \eqref{eq:energy}. The convolution
$\mu*\ind_T$ vanishes outside $A+T$ and has integral $|T|$.
Tonelli's theorem and Cauchy--Schwarz therefore give
\begin{equation}\label{eq:cs}
 e(A,T)=\int_{\R^d}(\mu*\ind_T)(x)^2\,dx
 \geq\frac{|T|^2}{|A+T|}>0.
\end{equation}
We use the standard variational inequality for minimum-energy
measures~\cite[Theorem~2.4]{Fuglede}, whose short proof we include.
Since $g_T(0)=|T|<\infty$, every Dirac mass has finite energy.
For each $a\in A$, the measure $(1-t)\mu+t\delta_a$ is an admissible
competitor for $0\leq t\leq1$. The right derivative of its energy at $t=0$ is
\[
 2\left(\int_A g_T(a-b)\,d\mu(b)-e(A,T)\right).
\]
Since $\mu$ minimizes the energy, this derivative is nonnegative. Thus
\begin{equation}\label{eq:potential}
 \int_A g_T(a-b)\,d\mu(b)\geq e(A,T)\qquad(a\in A).
\end{equation}
To turn this inequality into a covering measure, define $\nu$ by
\[
 \frac{d\nu}{dy}(y)=\frac{(\mu*\ind_{-T})(y)}{e(A,T)}.
\]
The reflection $-T$ ensures that convolution with $\ind_T$ recovers the
covariogram. Indeed, Tonelli's theorem and \eqref{eq:potential} give,
for every $a\in A$,
\begin{align*}
 (\nu*\ind_T)(a)
 &=\frac1{e(A,T)}\int_A\int_{\R^d}
       \ind_{-T}(y-b)\ind_T(a-y)\,dy\,d\mu(b)\\
 &=\frac1{e(A,T)}\int_A\int_{\R^d}
       \ind_T(x-(a-b))\ind_T(x)\,dx\,d\mu(b)\\
 &=\frac1{e(A,T)}\int_A g_T(a-b)\,d\mu(b)\geq1.
\end{align*}
Here the second equality uses $x=a-y$ and
$\ind_{-T}(y-b)=\ind_T(b-y)$.
The measure is supported on $A-T$ and has mass $|T|/e(A,T)$, so \eqref{eq:cs} gives the claimed bound.
\end{proof}

\section{Fractional illumination and the equality case}

Artstein-Avidan and Slomka observed that the equivalence between covering
and illumination extends to their fractional versions
\cite{ArtsteinSlomka}.
We use the following discretization, allowing a small enlargement of
the covering copies.

\begin{lemma}\label{lem:finite}
Let $0\in\intr K$ and $0<\lambda<\rho<1$. If a finite compactly supported
nonnegative measure $\nu$ covers $K$ by translates of $\lambda K$, there is a finite
weighted cover of $K$ by translates of $\rho K$ with the same total weight.
In particular, $\If(K)\leq\nu(\R^d)$.
\end{lemma}

\begin{proof}
Choose $r>0$ with $rB_2^d\subset K$. Convexity gives
$\lambda K+(\rho-\lambda)rB_2^d\subseteq\rho K\subset\intr K$.
Partition the compact support of $\nu$ into finitely many nonempty Borel
sets $C_j$ of diameter less than $(\rho-\lambda)r$. Choose $z_j\in C_j$ and put
$w_j=\nu(C_j)$. If $y\in C_j$ and $x-y\in\lambda K$, then
$x-z_j\in\rho K$. Thus, for every $x\in K$,
\[
 \sum_jw_j\ind_{\rho K}(x-z_j)
 \geq\int\ind_{\lambda K}(x-y)\,d\nu(y)\geq1.
\]
If $x\in\partial K\cap(z_j+\rho K)$, then $x-z_j\in\intr K$, so
$z_j\ne0$ and $-z_j/\|z_j\|_2$ illuminates $x$.
Omit zero centers, whose translates miss $\partial K$, and give each remaining
direction weight $w_j$.
\end{proof}

Translate $K$ so that $0\in\intr K$. For $0<\lambda<1$, apply
Proposition~\ref{prop:cover} with $A=K$ and $T=\lambda K$.
Since $K+\lambda K=(1+\lambda)K$, homogeneity of volume gives
a covering measure $\nu$ with
\[
 \nu(\R^d)\leq\frac{|K+\lambda K|}{|\lambda K|}
 =\left(\frac{1+\lambda}{\lambda}\right)^d.
\]
Lemma~\ref{lem:finite} gives $\If(K)\leq\nu(\R^d)$.
Letting $\lambda\uparrow1$, we obtain
\begin{equation}\label{eq:fractional-bound}
 \If(K)\leq2^d.
\end{equation}

It remains to determine when equality holds.
A point $v\in K$ is \emph{exposed} if some nonzero linear functional has
$v$ as its unique maximizer on $K$. The \emph{support function} of $K$ is
$h_K(u)=\max_{x\in K}\langle u,x\rangle$. Write $\conv A$ for the set of
finite convex combinations of points of $A$.

A finite set $V$ is
\emph{antipodal} if each pair of distinct points lies on two distinct
parallel supporting hyperplanes of $\conv V$. A hyperplane \emph{supports}
a convex set if it meets the set and the set lies in one of the closed
half-spaces it defines. We use the
Danzer--Gr\"unbaum theorem: an antipodal set in $\R^d$ has at most $2^d$
points, and equality holds exactly for the vertex sets of
parallelotopes~\cite{DanzerGrunbaum}.

\begin{proposition}\label{prop:rigidity}
If $e(K,K)=2^{-d}|K|$, then $K$ is a parallelotope.
\end{proposition}

\begin{proof}
Let $\mu\in\Prob(K)$ attain the minimum.
Equality in Cauchy--Schwarz in \eqref{eq:cs} gives
\begin{equation}\label{eq:flat}
 f:=\mu*\ind_K=2^{-d}\ind_{2K}\qquad\text{almost everywhere}.
\end{equation}
We first show that every exposed point carries mass $2^{-d}$.
If $v$ is exposed by a unit vector $u$, then every cap
$\{x\in2K:\langle u,x\rangle>2h_K(u)-\varepsilon\}$, $\varepsilon>0$, has
positive volume. On this cap,
\[
 f(x)\leq\mu\{a\in K:\langle u,a\rangle>h_K(u)-\varepsilon\},
\]
because $x-a\in K$ implies $\langle u,a\rangle>h_K(u)-\varepsilon$.
Since the cap has positive volume, \eqref{eq:flat} holds at some point of it,
so the set on the right has $\mu$-mass at least $2^{-d}$.
As $\varepsilon\downarrow0$, these sets decrease to $\{v\}$,
so continuity of measure from above gives $\mu\{v\}\geq2^{-d}$. Conversely,
$f\geq\mu\{v\}\ind_{K+v}$, and $K+v\subseteq2K$ has positive volume,
so \eqref{eq:flat} gives $\mu\{v\}\leq2^{-d}$.

The set $V$ of exposed points satisfies $\#V\leq2^d$, since $\mu(K)=1$.
By Straszewicz's theorem~\cite{Straszewicz}, a convex body is the closed
convex hull of its exposed points. Since $V$ is finite, $K=\conv V$,
so $K$ is a polytope with vertex set $V$.

We next show that $\mu$ has no mass away from these vertices. Write
\[
 \mu=2^{-d}\sum_{v\in V}\delta_v+\sigma,\qquad\sigma\geq0.
\]
The interiors of $K+v$ and $K+w$ are disjoint for distinct $v,w\in V$,
since an intersection would have positive volume and give $f\geq2^{1-d}$
there, contradicting \eqref{eq:flat}.
On $K+v$ the atom at $v$ already contributes $2^{-d}$, so all remaining
contributions vanish almost everywhere. Tonelli's theorem yields
\begin{equation}\label{eq:residual}
 \int_K |(K+a)\cap(K+v)|\,d\sigma(a)=0\qquad(v\in V).
\end{equation}
For every $a\in K$, at least one of these overlaps has positive volume.
Indeed, a convex combination of the vertices writes
$a=tv+(1-t)b$ with $v\in V$, $b\in K$, and $0<t\leq1$.
For any $p\in\intr K$, both
$x=(1-t)b+tp$ and $y=(1-t)v+tp$ lie in $\intr K$, and $x-y=a-v$.
Hence $a+y=v+x$ lies in the interiors of both $K+a$ and $K+v$.
Summing \eqref{eq:residual} over $v\in V$ integrates a strictly positive
function against $\sigma$ to zero, so $\sigma=0$.
Since $\mu$ is a probability measure, $2^{-d}\#V=1$, so $\#V=2^d$.

Finally, we show that $V$ is antipodal. Every $z\in\intr(K-K)$ can be
written as $z=s(x-y)$ with $x,y\in K$ and $0\leq s<1$, because
$0\in\intr(K-K)$. For $p\in\intr K$,
\[
 z=\bigl(sx+(1-s)p\bigr)-\bigl(sy+(1-s)p\bigr)
 \in\intr K-\intr K.
\]
Consequently, if $v-w$ were in $\intr(K-K)$, the interiors of $K+v$ and
$K+w$ would intersect. Their disjointness therefore implies
$v-w\in\partial(K-K)$. A supporting functional $u\ne0$ at $v-w$ satisfies
\[
 \langle u,v-w\rangle=h_{K-K}(u)=h_K(u)+h_K(-u).
\]
Since $v,w\in K$, equality forces
$\langle u,v\rangle=h_K(u)$ and $\langle u,w\rangle=-h_K(-u)$.
Thus $v$ and $w$ lie on opposite supporting hyperplanes of $K$.
The set $V$ is antipodal and has $2^d$ points, so the
Danzer--Gr\"unbaum theorem makes $K=\conv V$ a parallelotope.
\end{proof}

\begin{proof}[Proof of Theorem~\ref{thm:fractional}]
For the equality case we use the sharper energy estimate from
Proposition~\ref{prop:cover}. Together with Lemma~\ref{lem:finite}, it gives
\[
 \If(K)\leq\frac{|\lambda K|}{e(K,\lambda K)}\qquad(0<\lambda<1).
\]
Comparing the energies of all $\mu\in\Prob(K)$, and using
$\lambda K\subset K$, we have
\begin{align*}
 |e(K,\lambda K)-e(K,K)|
 &\leq\|g_{\lambda K}-g_K\|_\infty\\
 &\leq2\|\ind_{\lambda K}-\ind_K\|_{L^1}
 =2(1-\lambda^d)|K|\longrightarrow0.
\end{align*}
Letting $\lambda\uparrow1$ therefore gives $\If(K)\leq|K|/e(K,K)$.
If $K$ is not a parallelotope, Proposition~\ref{prop:rigidity} and
\eqref{eq:cs} imply $e(K,K)>2^{-d}|K|$, and hence $\If(K)<2^d$.

By affine invariance, it remains to consider the cube
$[-1,1]^d$. A direction illuminating a vertex
$\varepsilon\in\{-1,1\}^d$ must satisfy $\varepsilon_i u_i<0$ for every
$i$, so it illuminates at most one vertex. Summing the vertex constraints
gives total weight at least $2^d$. Conversely, for each boundary point $x$,
choose $\varepsilon_i=x_i$ whenever $|x_i|=1$, with the remaining signs
arbitrary. Then $-\varepsilon/\sqrt d$ illuminates $x$, so the $2^d$ such
directions, each with weight one, suffice.
Thus $\If([-1,1]^d)=2^d$.
\end{proof}

\section{From covering measures to finite covers}

We use the classical greedy covering argument (see Lov\'asz~\cite{Lovasz}
and Stein~\cite{SteinCover}). For its application to finite nets in geometric
covering, see Nasz\'odi~\cite[Theorem~1.2]{NaszodiCovering}.

\begin{lemma}\label{lem:greedy}
Let $A$ consist of $n\geq1$ points and let $T$ be compact. If $A$ admits a covering measure by translates of $T$ of mass at most $F>1$, then
\[
 N(A,T)\leq1+\lceil F\log n\rceil.
\]
\end{lemma}

\begin{proof}
Let $\nu$ be such a covering measure, and let $R\subseteq A$ be the set of
points not yet covered. For every $x\in R$, the covering property gives
\[
    \int \ind_T(x-z)\,d\nu(z)\geq1.
\]
Summing over $x\in R$ and interchanging the finite sum with the integral,
\[
    \int \#(R\cap(z+T))\,d\nu(z)\geq\#R.
\]
Since $\nu(\R^d)\leq F$, some translate $z+T$ therefore contains at least
$\#R/F$ points of $R$.

Starting with $R_0=A$, choose at each step such a translate and let
$R_{j+1}$ be the points of $R_j$ left uncovered. Then
\[
 \#R_{j+1}\leq\left(1-\frac1F\right)\#R_j,
\]
and hence
\[
 \#R_k\leq n\left(1-\frac1F\right)^k\leq ne^{-k/F}.
\]
For $k=\lceil F\log n\rceil$, this is at most one, so at most one further
translate is needed.
\end{proof}

\begin{proposition}\label{prop:hadwiger}
For every convex body $K\subset\R^d$ and $0<\delta<1/2$,
\begin{equation}\label{eq:explicit}
 N\bigl(K,(1-\delta)K\bigr)
 \leq1+\left\lceil
 d\left(\frac{2-2\delta}{1-2\delta}\right)^d
 \log\left(\frac4\delta+2\right)\right\rceil.
\end{equation}
\end{proposition}

\begin{proof}
We first find a large centrally symmetric subset of $K$, using Stein's
averaging argument~\cite[Theorem~4]{SteinSymmetry}. Fubini's theorem gives
\[
 \int_{\R^d}|K\cap(z-K)|\,dz=|K|^2.
\]
The integrand is supported on $2K$, so for some $z$ the intersection has
volume at least $2^{-d}|K|$. Recenter at $z/2$. Then
$S=K\cap(-K)$ is convex and centrally symmetric, with $|S|\geq2^{-d}|K|>0$.
Since $S$ has positive volume and is centrally symmetric,
$0\in\intr S\subseteq\intr K$.

Choose a maximal family $\{a+(\delta/2)S:a\in\Lambda\}$ with set of centers
$\Lambda\subset K$ and pairwise disjoint interiors. The packing estimate
below bounds its size, so successive additions terminate. Since $S=-S$,
maximality implies $K\subseteq\Lambda+\delta S$: any point outside this
set could be added as a new center. Thus $\Lambda$ is a finite net for $K$.
All packed translates lie in
$(1+\delta/2)K$, so
\begin{equation}\label{eq:net}
 \#\Lambda\leq
 \frac{(1+\delta/2)^d|K|}{(\delta/2)^d|S|}
 \leq\left(\frac4\delta+2\right)^d.
\end{equation}
Proposition~\ref{prop:cover} gives a covering measure for $\Lambda$ by
translates of $(1-2\delta)K$ with mass at most
\[
 \frac{|\Lambda+(1-2\delta)K|}{|(1-2\delta)K|}
 \leq\left(\frac{2-2\delta}{1-2\delta}\right)^d.
\]
Lemma~\ref{lem:greedy} and \eqref{eq:net} bound the number of translates
needed to cover $\Lambda$ by the right side of \eqref{eq:explicit}.
Since $K\subseteq\Lambda+\delta S$, adding $\delta S$ to each of these
translates covers $K$. Finally,
\[
 (1-2\delta)K+\delta S\subseteq(1-\delta)K,
\]
so this is a cover by translates of $(1-\delta)K$, as required.
\end{proof}

\begin{proof}[Proof of Theorem~\ref{thm:hadwiger}]
By definition, $H(K)\leq N(K,(1-\delta)K)$ for every $0<\delta<1/2$.
Thus \eqref{eq:explicit} bounds $H(K)$ uniformly over all convex bodies
$K\subset\R^d$.

For $d\geq3$, choose
\[
 \delta=\frac1{d\log d}.
\]
As $d\to\infty$, the two factors in \eqref{eq:explicit} satisfy
\begin{align*}
 \left(\frac{2-2\delta}{1-2\delta}\right)^d
 &=2^d\exp\bigl(d\delta+O(d\delta^2)\bigr)
  =2^d\left(1+O\left(\frac1{\log d}\right)\right),\\
 \log\left(\frac4\delta+2\right)
 &=\log(4d\log d+2)
  =\log d+\log\log d+O(1).
\end{align*}
Substitution in \eqref{eq:explicit} gives
\[
 H(d)\leq2^d\bigl(d\log d+d\log\log d+O(d)\bigr).
\]
The ceiling and the extra translate contribute at most two. Combining
this estimate with the standard cube lower bound $H(d)\geq2^d$ yields
$H(d)=2^{d+O(\log d)}$.
\end{proof}

\appendix
\section{AI use and discovery of the proof}\label{app:ai}

In the interest of transparency, as well as general curiosity towards the emerging AI tools applied to research mathematics, we are including a link to the full
\href{https://chatgpt.com/share/6aa156f6-e0f4-83eb-9532-bd258660e44b}{conversation with OpenAI's GPT-6 Pro} in which the ideas behind this paper were discovered.

There is an aspect of this exchange that is worth highlighting: the model was never directed to look up or prove the fractional illumination conjecture. The result seemingly emerged as an intermediate step towards a different goal set out for the LLM in the prompts. A later request to attack the full conjecture led to the greedy covering argument on a finite net used for the
asymptotic bound. Moreover, no proposed proof route or technical background for this problem was supplied as context before or during this chat.

Codex was subsequently used to revise the exposition and check sources.
The shared transcript demonstrates the exploratory process and includes
further claims beyond those developed here. The author takes full responsibility
for the mathematical claims and references in this manuscript.

\end{document}